\documentclass[a4paper,12pt]{amsart}

\usepackage{amsfonts}
\usepackage{amsmath}
\usepackage{amssymb}
\usepackage{graphicx}

\usepackage[usenames]{color}
\usepackage[colorlinks]{hyperref}

\newtheorem{thm}{Theorem}[section]
\newtheorem{lem}[thm]{Lemma}

\newtheorem{assum}[thm]{Assumption}

\theoremstyle{definition}

\theoremstyle{remark}
\newtheorem{rem}{Remark}[section]
\newtheorem{defn}{Definition}
\newtheorem{exam}[rem]{Example}
\numberwithin{equation}{section}

\begin{document}

\title[Fractional Klein-Gordon equation with singular mass]{Fractional Klein-Gordon equation with singular mass}


\author[A. Altybay]{Arshyn Altybay}
\address{
  Arshyn Altybay:
  \endgraf
  Department of Mathematics: Analysis, Logic and Discrete Mathematics
  \endgraf
  Ghent University, Belgium
  \endgraf
   and
  \endgraf
   Institute of Mathematics and Mathematical Modeling
  \endgraf
   Al-Farabi Kazakh National University
  \endgraf
  Almaty, Kazakhstan
  \endgraf
  {\it E-mail address} {\rm arshyn.altybay@gmail.com}
 }

\author[M. Ruzhansky]{Michael Ruzhansky}
\address{
	Michael Ruzhansky:
	 \endgraf
  Department of Mathematics: Analysis, Logic and Discrete Mathematics
  \endgraf
  Ghent University, Belgium
  \endgraf
  and
  \endgraf
  School of Mathematical Sciences
    \endgraf
    Queen Mary University of London
  \endgraf
  United Kingdom
	\endgraf
  {\it E-mail address} {\rm michael.ruzhansky@ugent.be}
}

\author[M. Sebih]{Mohammed Elamine Sebih}
\address{
  Mohammed Elamine Sebih:
    \endgraf
  Laboratory of Analysis and Control of Partial Differential Equations
  \endgraf
  Djillali Liabes University
  \endgraf
  Sidi Bel Abbes, Algeria
  \endgraf
  {\it E-mail address} {\rm sebihmed@gmail.com}
 }

\author[N. Tokmagambetov]{Niyaz Tokmagambetov}
\address{
  Niyaz Tokmagambetov:
    \endgraf
  Department of Mathematics: Analysis, Logic and Discrete Mathematics
  \endgraf
  Ghent University, Belgium
  \endgraf
   and
  \endgraf
  Al-Farabi Kazakh National University
  \endgraf
  Almaty, Kazakhstan
  \endgraf
  {\it E-mail address} {\rm tokmagambetov@math.kz}
 }

\thanks{This research was funded by the Science Committee of the Ministry of Education and Science of the Republic of Kazakhstan (Grant No. AP09058069) and by the FWO Odysseus 1 grant G.0H94.18N: Analysis and Partial Differential Equations. Michael Ruzhansky was supported in parts by the EPSRC Grant EP/R003025/1. Arshyn Altybay was supported in parts by the MESRK Grant AP08052028 of the Science Committee of the Ministry of Education and Science of the Republic of Kazakhstan. Mohammed Sebih was supported by the Algerian Scholarship P.N.E. 2018/2019 during his visit to the University of Stuttgart and Ghent University. Also, Mohammed Sebih thanks Professor Jens Wirth and Professor Michael Ruzhansky for their warm hospitality.}


\begin{abstract}
We consider a space-fractional wave equation with a singular mass term depending on the position and prove that it is very weak well-posed. The uniqueness is proved in some appropriate sense. Moreover, we prove the consistency of the very weak solution with classical solutions when they exist. In order to study the behaviour of the very weak solution near the singularities of the coefficient, some numerical experiments are conducted where the appearance of a wall effect for the singular masses of the strength of $\delta^2$ is observed.
\end{abstract}

\maketitle

\section{Introduction}
In this work we investigate the Klein-Gordon equation with a non-negative singular mass term depending on the spacial variable. We use the fractional Laplacian instead of the classical one and prove that the problem has a very weak solution.

The concept of very weak solutions was introduced in \cite{GR} for the analysis of second order hyperbolic equations with non-regular time-dependent coefficients and was applied to several physical models in the papers \cite{RT17a}, \cite{RT17b}. We also refer to \cite{MRT19}, where the authors showed the well posedness in the very weak sense of the damped wave equation with an irregular time-dependent dissipation. Our aim here is to apply this notion to the Klein-Gordon equation with a singular mass term.

The classical Klein-Gordon equation with a constant mass term and generated by the classical Laplacian describes the dynamics of spinless particles. With the aim to describe the propagation of the spinless massive field in curved spacetime in general relativity, the Klein-Gordon equation has been generalized along two lines: one is to include a position-dependent mass term \cite{AST}, \cite{DJ}, \cite{WLLW06}, \cite{WLLW18}, and the second is to introduce a fractional Laplacian \cite{GBSD}, \cite{GGB}. On the other hand in the microscopic scale, e.g. in the theory of fluids, the mass behaves like distributions \cite{B}, which motivates our intention to allow it to be discontinuous or less regular.

The generalized Klein-Gordon equation has widespread applications such as in the theory of quantum gravity and quantum liquids to name only two (see \cite{M} and the references therein).

\section{Main results}
For $\alpha >0$ and $d\in\mathbb N$, we investigate the Cauchy problem for the fractional Klein-Gordon equation
\begin{equation}
\left\lbrace
\begin{array}{l}
u_{tt}(t,x)+(-\Delta)^{\alpha} u(t,x) + m(x)u(t,x)=0 ,~~~(t,x)\in\left[0,T\right]\times \mathbb{R}^{d},\\
u(0,x)=u_{0}(x), \,\,\, u_{t}(0,x)=u_{1}(x), \,\,\, x\in\mathbb{R}^{d}. \label{Equation}
\end{array}
\right.
\end{equation}
Here, the function $m$ is supposed to be non-negative and singular. In the regular situation, i.e. in the case when the coefficient $m$ is a regular function we have the following lemma.

To start with, let us define some notions and notations that we use throughout this paper. Firstly, the notation $f\lesssim g$ means that there exists a positive constant $C$ such that $f \leq Cg$.
Secondly, the fractional Sobolev space $H^{\alpha}(\mathbb{R}^{d})$ is defined as follows :
$
H^{\alpha}(\mathbb{R}^{d})=\big\{ u\in L^{2}(\mathbb{R}^{d}): \Vert u\Vert_{H^{\alpha}} < +\infty \big\}
$, where $\Vert u\Vert_{H^{\alpha}}:=\Vert u\Vert_{L^2}+\Vert (-\Delta)^{\frac{\alpha}{2}}u\Vert_{L^2}$.\\
We will also use the following notation :
\begin{equation*}
\Vert u(t,\cdot)\Vert := \Vert u(t,\cdot)\Vert_{H^{\alpha}} + \Vert \partial_{t}u(t,\cdot)\Vert_{L^2}.
\end{equation*}

\begin{lem}\label{lem-est}
Let $m\in L^{\infty}(\mathbb{R}^d)$ and $m\geq 0$. Suppose that $u_0\in H^{\alpha}(\mathbb R^d)$ and $u_1\in L^{2}(\mathbb R^d)$. Then, there is a unique solution $u\in C([0, T]; H^{\alpha}(\mathbb R^d))\cap C^{1}([0, T]; L^{2}(\mathbb R^d))$ to (\ref{Equation}), and it satisfies the estimate
\begin{equation}
\Vert u(t,\cdot)\Vert^{2} \lesssim \left(1+\Vert m\Vert_{L^{\infty}}\right)\left[\Vert u_{1}\Vert_{L^2}^{2}+\Vert u_{0}\Vert_{H^{\alpha}}^{2}\right]. \label{Energy estimate}
\end{equation}

\end{lem}

\begin{proof}
Multiplying the equation (\ref{Equation}) on both sides by $u_t$ and integrating, we get
\begin{equation}
Re \left(\langle\partial_{t}^{2}u(t,\cdot),\partial_{t}u(t,\cdot)\rangle_{L^2} + \langle(-\Delta)^{\alpha}u(t,\cdot),\partial_{t}u(t,\cdot)\rangle_{L^2} + \langle m(\cdot)u(t,\cdot),\partial_{t}u(t,\cdot)\rangle_{L^2} \right)=0,
\label{Energy functional}
\end{equation}
where $\langle \cdot, \cdot \rangle_{L^2}$ is the inner product of $L^2(\mathbb{R}^{d})$.

Easy calculations show that
$
Re \langle\partial_{t}^{2}u(t,\cdot),\partial_{t}u(t,\cdot)\rangle_{L^2}=\frac{1}{2}\partial_{t}\langle\partial_{t}u(t,\cdot),\partial_{t}u(t,\cdot)\rangle_{L^2},
$
\begin{equation*}
Re \langle(-\Delta)^{\alpha}u(t,\cdot),\partial_{t}u(t,\cdot)\rangle_{L^2}=\frac{1}{2}\partial_{t}\langle(-\Delta)^{\frac{\alpha}{2}}u(t,\cdot),(-\Delta)^{\frac{\alpha}{2}}u(t,\cdot)\rangle_{L^2},
\end{equation*}
and
\begin{equation*}
Re \langle m(\cdot)u(t,\cdot),\partial_{t}u(t,\cdot)\rangle_{L^2}=\frac{1}{2}\partial_{t}\langle m^{\frac{1}{2}}(\cdot)u(t,\cdot),m^{\frac{1}{2}}(\cdot)u(t,\cdot)\rangle_{L^2}.
\end{equation*}

Let us denote by
$
E(t):=\Vert\partial_{t}u(t,\cdot)\Vert_{L^2}^{2}+\Vert(-\Delta)^{\frac{\alpha}{2}}u(t,\cdot)\Vert_{L^2}^{2}+\Vert m^{\frac{1}{2}}(\cdot)u(t,\cdot)\Vert_{L^2}^{2},
$
the energy functional of the system (\ref{Equation}).
From (\ref{Energy functional}) it follows that $\partial_{t} E(t)=0$, and thus
$E(t)=E(0).$
By taking in consideration that $\Vert m^{\frac{1}{2}} \, u_{0}\Vert_{L^2}^{2}$ can be estimated by
$
\Vert m^{\frac{1}{2}}\, u_{0} \Vert_{L^2}^{2}\leq \Vert m \, \Vert_{L^{\infty}}\Vert u_{0} \Vert_{L^2}^{2},
$
it follows that
\begin{equation}
\Vert \partial_{t}u(t,\cdot)\Vert_{L^2}^{2}\lesssim\left(\Vert u_{1}\Vert_{L^2}^{2}+\Vert (-\Delta)^{\frac{\alpha}{2}}u_{0}\Vert_{L^2}^{2}+\Vert m\Vert_{L^{\infty}}\Vert u_{0}\Vert_{L^2}^{2}\right), \label{Estimate for d_t u}
\end{equation}
\begin{equation}
\Vert (-\Delta)^{\frac{\alpha}{2}} u(t,\cdot)\Vert_{L^2}^{2} \lesssim \left(\Vert u_{1}\Vert_{L^2}^{2}+\Vert (-\Delta)^{\frac{\alpha}{2}}u_{0}\Vert_{L^2}^{2}+\Vert m\Vert_{L^{\infty}}\Vert u_{0}\Vert_{L^2}^{2}\right), \label{Estimate for Lapl u}
\end{equation}
and
\begin{equation}
\Vert m^{\frac{1}{2}}(\cdot)u(t,\cdot)\Vert_{L^2}^{2} \lesssim \left(\Vert u_{1}\Vert_{L^2}^{2}+\Vert (-\Delta)^{\frac{\alpha}{2}}u_{0}\Vert_{L^2}^{2}+\Vert m\Vert_{L^{\infty}}\Vert u_{0}\Vert_{L^2}^{2}\right). \label{Estimate for mu}
\end{equation}
Hence, the desired estimates for $\partial_{t}u(t,\cdot)$ and $(-\Delta)^{\frac{\alpha}{2}} u(t,\cdot)$ are proved. Let us now estimate $u$. Applying the Fourier transform to (\ref{Equation}), the problem can be rewritten as a second order ordinary differential equation \begin{equation}\label{Equation Fourier transf}
\hat{u}_{tt}(t,\xi)+\vert \xi\vert^{2\alpha}\hat{u}(t,\xi)=\hat{f}(t,\xi),
\end{equation}
with the initial conditions $\hat{u}(0,\xi)=\hat{u}_{0}(\xi)$ and $\hat{u}_{t}(0,\xi)=\hat{u}_{1}(\xi)$. Here $\hat{f}$, $\hat{u}$, denote the Fourier transform of $f$ and $u$ in the spacial variable and
$
f(t,x):=-m(x)u(t,x).
$
We note that in (\ref{Equation Fourier transf}), we see $\hat{f}$ as a source term.

By solving first the homogeneous equation and by application of Duhamel's principle (see, e.g. \cite{Eva98}), we get the following representation of the solution
\begin{equation}
\hat{u}(t,\xi)=\cos (t\vert\xi\vert^{\alpha}) \hat{u}_{0}(\xi) + \frac{\sin (t\vert\xi\vert^{\alpha})}{\vert\xi\vert^{\alpha}} \hat{u}_{1}(\xi) + \int_{0}^{t}\frac{\sin ((t-s)\vert\xi\vert^{\alpha})}{\vert\xi\vert^{\alpha}} \hat{f}(s,\xi) ds. \label{Representation of sol}
\end{equation}
Taking the $L^{2}$ norm in (\ref{Representation of sol}) and using the following estimates: 1) $\vert \cos (t\vert\xi\vert^{\alpha})\vert \leq 1$, for $t\in \left[0,T\right]$ and $\xi\in \mathbb{R}^{d}$,
2) $\vert \sin (t\vert\xi\vert^{\alpha})\vert \leq 1$, for large frequencies and $t\in \left[0,T\right]$ and, 3) $\vert \sin (t\vert\xi\vert^{\alpha})\vert \leq t\vert\xi\vert^{\alpha} \leq T\vert\xi\vert^{\alpha}$, for small frequencies and $t\in \left[0,T\right]$,
we get that
\begin{equation*}
\Vert \hat{u}(t,\cdot)\Vert_{L^2}^{2} \lesssim \Vert \hat{u}_{0}\Vert_{L^2}^{2} + \Vert \hat{u}_{1}\Vert_{L^2}^{2} + \int_{0}^{t}\Vert \hat{f}(s,\cdot)\Vert_{L^2}^{2}ds.
\end{equation*}
By Parseval-Plancherel formula we arrive at
\begin{equation*}
\Vert u(t,\cdot)\Vert_{L^2}^{2} \lesssim \Vert u_{0}\Vert_{L^2}^{2} + \Vert u_{1}\Vert_{L^2}^{2} + \int_{0}^{T}\Vert m(\cdot)u(s,\cdot)\Vert_{L^2}^{2}ds.
\end{equation*}
Using the estimate (\ref{Estimate for mu}) and taking in consideration that the last term in the above estimate can be estimated by
$
\Vert m(\cdot)u(t,\cdot)\Vert_{L^{2}} \leq \Vert m \Vert_{L^{\infty}}^{\frac{1}{2}}\Vert m^{\frac{1}{2}}u(t,\cdot)\Vert_{L^{2}},
$
we get
\begin{equation}
\Vert u(t,\cdot)\Vert_{L^2}^{2} \lesssim \left(1+\Vert m\Vert_{L^{\infty}}\right) \left[\Vert u_{0}\Vert_{H^{\alpha}}^{2}+\Vert u_{1}\Vert_{L^2}^{2}\right]. \label{Estimate for u}
\end{equation}
The estimate (\ref{Energy estimate}) follows by summing the estimates (\ref{Estimate for d_t u}), (\ref{Estimate for Lapl u}) and (\ref{Estimate for u}),  ending the proof. \end{proof}





\subsection{Very weak solutions: Existence} Here, we consider an irregular case when the mass term $m$ of the equation \eqref{Equation} has strong singularities, namely, $\delta$-function or "$\delta^{2}$-function" like behaviours. In what follows, we will understand a multiplication of distributions in the sense of the Colombeau algebra \cite{Obe92}.

Now we introduce a notion of the very weak solution to the Cauchy problem \eqref{Equation} and prove the existence result. We start by regularising the coefficient $m$ using a suitable mollifier $\psi$ generating families of smooth functions $(m_{\varepsilon})_{\varepsilon}$, namely,
$
m_{\varepsilon}(x)=m\ast \psi_{\varepsilon }(x),
$
where
$
\psi_{\varepsilon}(x)=\varepsilon^{-d}\psi(x/\varepsilon)
$
and $\varepsilon\in\left(0,1\right]$. The function $\psi$ is a Friedrichs-mollifier, i.e. $\psi\in C_{0}^{\infty}(\mathbb{R}^{d})$, $\psi\geq 0$ and $\int\psi =1$.

\begin{assum}\label{assum}
We make the following assumption on the regularisation $(m_{\varepsilon})_{\varepsilon}$ of the coefficient $m$: there exist $N_{0}\in \mathbb{N}_{0}$ and $C>0$ such that
\begin{equation}
\Vert m_{\varepsilon}\Vert_{L^{\infty}}\leq C\varepsilon^{-N_{0}},
\label{Moderetness hyp coeff}
\end{equation}
for all $\varepsilon\in(0, 1].$
\end{assum}
We note that by the structure theorems of distributions, such assumption is natural and is satisfied, e.g, for $m\in \mathcal{D}^{\prime}$. Let us give some examples.
\begin{exam}
Let $m(x)=\delta_{0}(x)$. Then, we have
$m_{\varepsilon}(x)=m\ast\psi_{\varepsilon}(x) =\varepsilon^{-d}\psi(\varepsilon^{-1}x)\leq C\varepsilon^{-d}.$ Moreover, for $m(x)=\delta_{0}^{2}(x)$, one can define $m_{\varepsilon}(x) =\varepsilon^{-2d}\psi^{2}(\varepsilon^{-1}x) \leq C\varepsilon^{-2d}.$
\end{exam}

\begin{defn}[Moderateness] \label{Def:Moderetness}
{\rm (i)} A net of functions $(g_{\varepsilon})_{\varepsilon}$ is said to be ${L^{\infty}}$-moderate, if there exist $N\in\mathbb{N}_{0}$ and $c>0$ such that
\begin{equation*}
\Vert g_{\varepsilon}\Vert_{L^{\infty}} \leq c\varepsilon^{-N}.
\end{equation*}
{\rm (ii)} A net of functions $(u_{\varepsilon})_{\varepsilon}$ from $C([0, T]; H^{\alpha})\cap C^{1}([0, T]; L^{2})$ is said to be $C^{1}$-moderate, if there exist $N\in\mathbb{N}_{0}$ and $c>0$ such that
\begin{equation*}
\sup_{t\in[0, T]}\Vert u_{\varepsilon}(t, \cdot)\Vert \leq c\varepsilon^{-N}.
\end{equation*}
\end{defn}

\begin{rem}
By the assumption (\ref{Moderetness hyp coeff}), $m_{\varepsilon}$ is ${L^{\infty}}$-moderate in the sense of the last definition.
\end{rem}

\begin{defn}[Very Weak Solution]
Let $(u_{0},u_{1})\in H^{\alpha}(\mathbb{R}^{d})\times L^{2}(\mathbb{R}^{d})$. Then the net  $(u_{\varepsilon})_{\varepsilon}\in C([0, T]; H^{\alpha}(\mathbb R^{d}))\cap C^{1}([0, T]; L^{2}(\mathbb R^{d}))$ is a very weak solution to the Cauchy problem (\ref{Equation}) if there exists an ${L^{\infty}}$-moderate regularisation $(m_{\varepsilon})_{\varepsilon}$ of the coefficient $m$ such that $(u_{\varepsilon})_{\varepsilon}$ solves the regularized problem
\begin{equation}
\left\lbrace
\begin{array}{l}
\partial_{t}^{2}u_{\varepsilon}(t,x)+(-\Delta)^{\alpha} u_{\varepsilon}(t,x) + m_{\varepsilon}(x)u_{\varepsilon}(t,x)=0 ,~~~(t,x)\in\left[0,T\right]\times \mathbb{R}^{d},\\
u_{\varepsilon}(0,x)=u_{0}(x),\partial_{t}u_{\varepsilon}(0,x)=u_{1}(x),\,\,\, x\in\mathbb{R}^{d},
\label{Regularized equation}
\end{array}
\right.
\end{equation}
for all $\varepsilon\in\left(0,1\right]$, and is $C^{1}$-moderate.
\end{defn}

\begin{thm}
Assume that the regularisation $(m_{\varepsilon})_{\varepsilon}$ of the coefficient $m$ satisfies the moderateness condition (\ref{Moderetness hyp coeff}). Then the Cauchy problem (\ref{Equation}) has a very weak solution.
\end{thm}

\begin{proof}
Since $u_0$ and $u_1$ are smooth enough, using the moderateness assumption (\ref{Moderetness hyp coeff}) and the energy estimate \eqref{Energy estimate}, we arrive at
\begin{equation*}
\Vert u_{\varepsilon}\Vert\leq C\varepsilon^{-N_{0}/2},
\end{equation*}
where $N_0$ is from (\ref{Moderetness hyp coeff}), which means that $(u_{\varepsilon})_{\varepsilon}$ is $C^{1}$-moderate.
\end{proof}

\subsection{Uniqueness} We say that the Cauchy problem (\ref{Equation}) has a unique very weak solution, if for all families of regularisations $(m_{\varepsilon})_{\varepsilon}$ and $(\Tilde{m}_{\varepsilon})_{\varepsilon}$, of the coefficient $m$, satisfying
$
\Vert m_{\varepsilon}-\Tilde{m}_{\varepsilon}\Vert_{L^{\infty}}\leq C_{k}\varepsilon^{k} \text{~~for all~~} k>0,
$
it follows that
\begin{equation*}
\Vert u_{\varepsilon}(t,\cdot)-\Tilde{u}_{\varepsilon}(t,\cdot)\Vert_{L^{2}} \leq C_{N}\varepsilon^{N}
\end{equation*}
for all $N>0$, for all $t\in[0, T]$, where $(u_{\varepsilon})_{\varepsilon}$ and $(\Tilde{u}_{\varepsilon})_{\varepsilon}$ are the families of solutions corresponding to $(m_{\varepsilon})_{\varepsilon}$ and $(\Tilde{m}_{\varepsilon})_{\varepsilon}$, respectively.


\begin{thm}
Let $T>0$. Assume that $m\geq 0$ in the sense that its regularisations as functions are non-negative. Suppose that $(u_{0},u_{1})\in H^{\alpha}(\mathbb{R}^{d})\times L^{2}(\mathbb{R}^{d})$. Then, the very weak solution to the Cauchy problem (\ref{Equation}) is unique.

\end{thm}

\begin{proof}
Let $(u_{\varepsilon})_{\varepsilon}$ and $(\Tilde{u}_{\varepsilon})_{\varepsilon}$ be very weak solutions to the Cauchy problem (\ref{Equation}) corresponding to the coefficients $(m_{\varepsilon})_{\varepsilon}$ and $(\Tilde{m}_{\varepsilon})_{\varepsilon}$ and assume that
$
\Vert m_{\varepsilon}-\Tilde{m}_{\varepsilon}\Vert_{L^{\infty}}\leq C_{k}\varepsilon^{k} \text{~~for all~~} k>0.
$
Let us denote by $U_{\varepsilon}(t,x):=u_{\varepsilon}(t,x)-\Tilde{u}_{\varepsilon}(t,x)$, then, $U$ satisfies the equation
\begin{equation}
\left\lbrace
\begin{array}{l}
\partial_{t}^{2}U_{\varepsilon}(t,x)+(-\Delta)^{\alpha} U_{\varepsilon}(t,x) + m_{\varepsilon}(x)U_{\varepsilon}(t,x)=f_{\varepsilon}(t,x),\\
U(0,x)=0, ~~\partial_{t}U_{\varepsilon}(0,x)=0, \label{Equation uniqueness}
\end{array}
\right.
\end{equation}
with
$
f_{\varepsilon}(t,x)=(\Tilde{m}_{\varepsilon}(x)-m_{\varepsilon}(x))\Tilde{u}_{\varepsilon}(t,x).
$
Using Duhamel's principle, $U_{\varepsilon}$ is given by
$
U_{\varepsilon}(x,t)=\int_{0}^{t}V_{\varepsilon}(x,t-s;s)ds
$,
where $V_{\varepsilon}(x,t;s)$ solves the problem
\begin{equation*}
\left\lbrace
\begin{array}{l}
\partial_{t}^{2}V_{\varepsilon}(x,t;s)+(-\Delta)^{\alpha} V_{\varepsilon}(x,t;s) + m_{\varepsilon}(x)V_{\varepsilon}(x,t;s)=0,\\
V_{\varepsilon}(x,0;s)=0, ~\partial_{t}V_{\varepsilon}(x,0;s)=f_{\varepsilon}(s,x).
\end{array}
\right.
\end{equation*}
Taking $U_{\varepsilon}$ in $L^{2}$-norm and using (\ref{Energy estimate}) to get estimate for $V_{\varepsilon}$, we arrive at
\begin{equation*}
\begin{split}
\Vert U_{\varepsilon}(\cdot,t)\Vert_{L^2}  &\leq C \left(1+\Vert m_{\varepsilon}\Vert_{L^{\infty}}\right)^{\frac{1}{2}}\int_{0}^{T}\Vert f_{\varepsilon}(s,\cdot)\Vert_{L^2} ds \\
&\leq C \left(1+\Vert m_{\varepsilon}\Vert_{L^{\infty}}\right)^{\frac{1}{2}}\Vert \Tilde{m}_{\varepsilon}-m_{\varepsilon}\Vert_{L^{\infty}}\int_{0}^{T}\Vert \Tilde{u}_{\varepsilon}(s,\cdot)\Vert_{L^{2}} ds.
\end{split}
\end{equation*}
We have that
$
\Vert m_{\varepsilon}-\Tilde{m}_{\varepsilon}\Vert_{L^{\infty}}\leq C_{k}\varepsilon^{k} \text{~~for all~~} k>0,
$
 the net $(m_{\varepsilon})_{\varepsilon}$ is moderate by assumption and $(\Tilde{u}_{\varepsilon})_{\varepsilon}$ is moderate as a very weak solution to the Cauchy problem (\ref{Equation}). Then, for all $N>0$, we obtain
\begin{equation*}
\Vert U_{\varepsilon}(\cdot,t)\Vert_{L^2}=\Vert u_{\varepsilon}(t,\cdot)-\Tilde{u}_{\varepsilon}(t,\cdot)\Vert_{L^2} \lesssim \varepsilon^{N}.
\end{equation*}
Thus, the very weak solution is unique.
\end{proof}


\subsection{Consistency} We want to prove that in the case when a classical solution exists for the Cauchy problem (\ref{Equation}) as in Lemma \ref{lem-est}, the very weak solution recaptures the classical one.

\begin{thm}
Let $(u_{0},u_{1})\in H^{\alpha}(\mathbb{R}^{d})\times L^{2}(\mathbb{R}^{d})$. Assume that $m\in L^{\infty}(\mathbb{R}^{d})$ is non-negative and, let us consider the Cauchy problem
\begin{equation}
\left\lbrace
\begin{array}{l}
u_{tt}(t,x)+(-\Delta)^{\alpha} u(t,x) + m(x)u(t,x)=0 ,~~~(t,x)\in\left(0, T\right)\times \mathbb{R}^{d},\\
u(0,x)=u_{0}(x),u_{t}(0,x)=u_{1}(x), \,\,\, x\in\mathbb{R}^{d}. \label{Equation with reg. coeff}
\end{array}
\right.
\end{equation}
Let $(u_{\varepsilon})_{\varepsilon}$ be a very weak solution of (\ref{Equation with reg. coeff}). Then for any regularising family $m_{\varepsilon}=m\ast\psi_{\varepsilon}$, for any $\psi\in C_{0}^{\infty}$, $\psi\geq 0$, $\int\psi =1$, the net $(u_{\varepsilon})_{\varepsilon}$ converges to the classical solution of the Cauchy problem (\ref{Equation with reg. coeff}) in $L^{2}$ as $\varepsilon \rightarrow 0$.
\end{thm}

\begin{proof}
The classical solution satifies
\begin{equation*}
\left\lbrace
\begin{array}{l}
u_{tt}(t,x)+(-\Delta)^{\alpha} u(t,x) + m(x)u(t,x)=0 ,~~~(t,x)\in\left(0, T\right)\times \mathbb{R}^{d},\\
u(0,x)=u_{0}(x),u_{t}(0,x)=u_{1}(x), \,\,\, x\in\mathbb{R}^{d}.
\end{array}
\right.
\end{equation*}
For the very weak solution, there is a representation $(u_{\varepsilon})_{\varepsilon}$ such that
\begin{equation*}
\left\lbrace
\begin{array}{l}
\partial_{t}^{2}u_{\varepsilon}(t,x)+(-\Delta)^{\alpha} u_{\varepsilon}(t,x) + m_{\varepsilon}(x)u_{\varepsilon}(t,x)=0 ,~~~(t,x)\in\left(0, T\right)\times \mathbb{R}^{d},\\
u_{\varepsilon}(0,x)=u_{0}(x),\partial_{t}u_{\varepsilon}(0,x)=u_{1}(x), \,\,\, x\in\mathbb{R}^{d}.
\end{array}
\right.
\end{equation*}
Taking the difference of the above equations, we get
\begin{equation*}
\left\lbrace
\begin{array}{l}
\partial_{t}^{2}(u-u_{\varepsilon})(t,x)+(-\Delta)^{\alpha} (u-u_{\varepsilon})(t,x) + m_{\varepsilon}(x)(u-u_{\varepsilon})(t,x)=\eta_{\varepsilon}(t,x),\\
(u-u_{\varepsilon})(0,x)=0,~~\partial_{t}(u-u_{\varepsilon})(0,x)=0, \,\,\, x\in\mathbb{R}^{d},
\end{array}
\right.
\end{equation*}
where
$
\eta_{\varepsilon}(t,x)=(m(x)-m_{\varepsilon}(x))u(t,x).
$
Let us denote by $W_{\varepsilon}(t,x):=(u-u_{\varepsilon})(t,x)$. Once again, using Duhamel's principle, $W_{\varepsilon}$ is given by
$
W_{\varepsilon}(x,t)=\int_{0}^{t}V_{\varepsilon}(x,t-s;s)ds,
$
where $V_{\varepsilon}(x,t;s)$ solves the problem
\begin{equation*}
\left\lbrace
\begin{array}{l}
\partial_{t}^{2}V_{\varepsilon}(x,t;s)+(-\Delta)^{\alpha} V_{\varepsilon}(x,t;s) + m_{\varepsilon}(x)V_{\varepsilon}(x,t;s)=0,\\
V_{\varepsilon}(x,0;s)=0, ~\partial_{t}V_{\varepsilon}(x,0;s)=\eta_{\varepsilon}(s,x).
\end{array}
\right.
\end{equation*}
We have that
$
\Vert m-m_{\varepsilon}\Vert_{L^{\infty}} \rightarrow 0 \text{~~as~~} \varepsilon\rightarrow 0.
$
Taking the $L^{2}$-norm for $W_{\varepsilon}$ and using the energy estimate (\ref{Energy estimate}), we get
\begin{align*}
\Vert W_{\varepsilon}(\cdot,t)\Vert_{L^2} \leq \int_{0}^{T}\Vert V_{\varepsilon}(\cdot,t-s;s)\Vert_{L^2} ds
 \leq C \left(1+\Vert m_{\varepsilon}\Vert_{L^{\infty}}\right)^{1/2}\Vert m-m_{\varepsilon}\Vert_{L^{\infty}}^{1/2}\int_{0}^{T}\Vert u(s,\cdot)\Vert_{L^{2}} ds.
\end{align*}
Since $\Vert m_{\varepsilon}\Vert_{L^{\infty}}\leq C$ it follows that $(u_{\varepsilon})_{\varepsilon}$ converges to $u$ in $L^{2}$ as $\varepsilon\to0$.
\end{proof}

\section{Numerical experiments}

In this Section, we do some numerical experiments. Let us analyse our problem by regularising a distributional mass term $m(x)$ by a parameter $\varepsilon$. We define
$
m_\varepsilon (x):=(m\ast\varphi_\varepsilon)(x),
$
as the convolution with the mollifier
$\varphi_\varepsilon(x)=\frac{1}{\varepsilon} \varphi(x/\varepsilon),$
where
$
\varphi(x)=
\begin{cases}
c \exp{\left(\frac{1}{x^{2}-1}\right)}, |x| < 1, \\
0, \,\,\,\,\,\,\,\,\,\,\,\,\,\,\,\,\,\,\,\,\,\,\,\,\,\,\,\,  |x|\geq 1,
\end{cases}
$
with  $c \simeq 2.2523$ to have
$
\int\limits_{-\infty}^{\infty}  \varphi(x)dx=1.
$
Then, instead of \eqref{Equation} we consider the regularised problem
\begin{equation}\label{RE-01}
\partial^{2}_{t}u_{\varepsilon}(t,x)-\partial^{2}_{x} u_{\varepsilon}(t,x)+ m_{\varepsilon}(x) u_{\varepsilon}(t,x) =0, \; (t,x)\in[0,T]\times\mathbb R,
\end{equation}
with the initial data $u_\varepsilon(0,x)=u_0 (x)$ and $\partial_t u_\varepsilon(0,x)=u_1(x),$ for all $x\in\mathbb R.$ Here, we put
\begin{center}
$
u_0 (x)=
\begin{cases}
\exp{\left(\frac{1}{(x-50)^{2}-0.25}\right)}, \,\, |x-50| < 0.5, \\
0, \,\,\,\,\,\,\,\,\,\,\,\,\,\,\,\,\,\,\,\,\,\,\,\,\,\,\,\,\,\,\, \,\,\,\,\,\, \,\,\, |x-50| \geq 0.5,
\end{cases}
$
\end{center}
and $u_1 (x)\equiv0$. Note that ${\rm supp }\, u_0\subset[49.5, 50.5]$.

\begin{figure}[ht!]
\begin{minipage}[h]{0.29\linewidth}
\center{\includegraphics[scale=0.25]{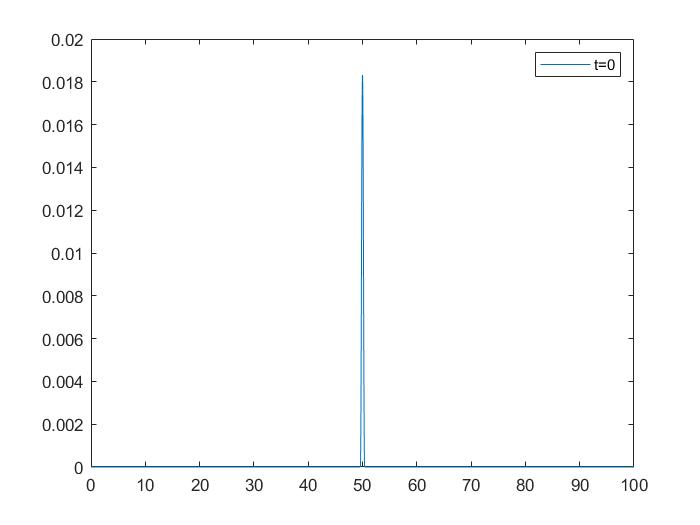}}
\end{minipage}
\hfill
\begin{minipage}[h]{0.29\linewidth}
\center{\includegraphics[scale=0.25]{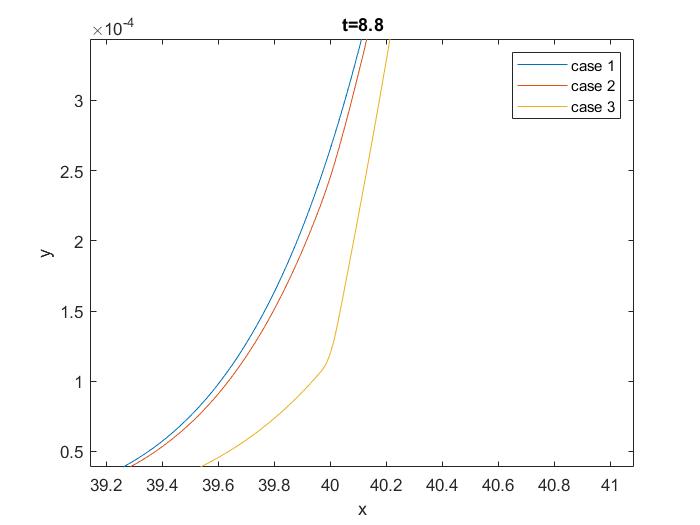}}
\end{minipage}
\hfill
\begin{minipage}[h]{0.29\linewidth}
\center{\includegraphics[scale=0.25]{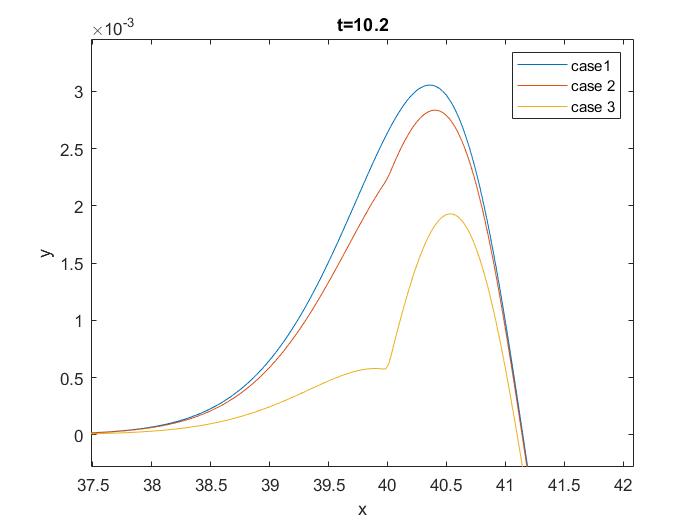}}
\end{minipage}
\hfill
\begin{minipage}[h]{0.29\linewidth}
\center{\includegraphics[scale=0.25]{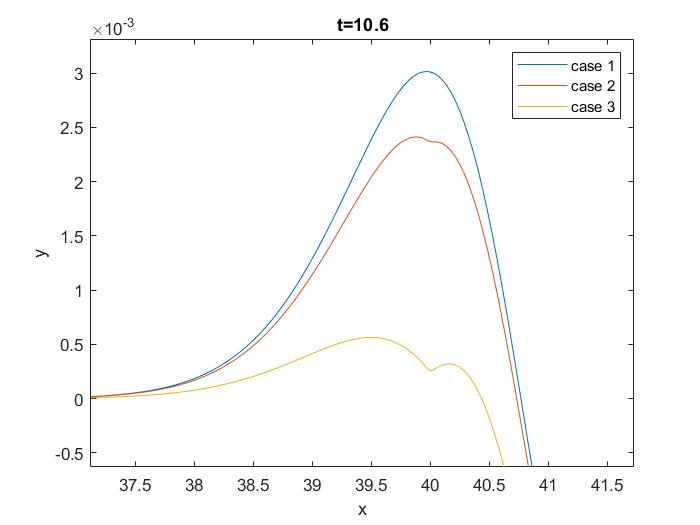}}
\end{minipage}
\hfill
\begin{minipage}[h]{0.29\linewidth}
\center{\includegraphics[scale=0.25]{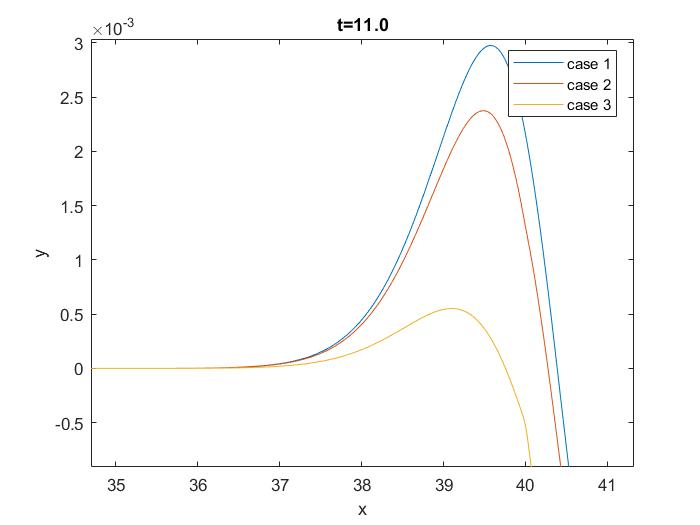}}
\end{minipage}
\hfill
\begin{minipage}[h]{0.29\linewidth}
\center{\includegraphics[scale=0.25]{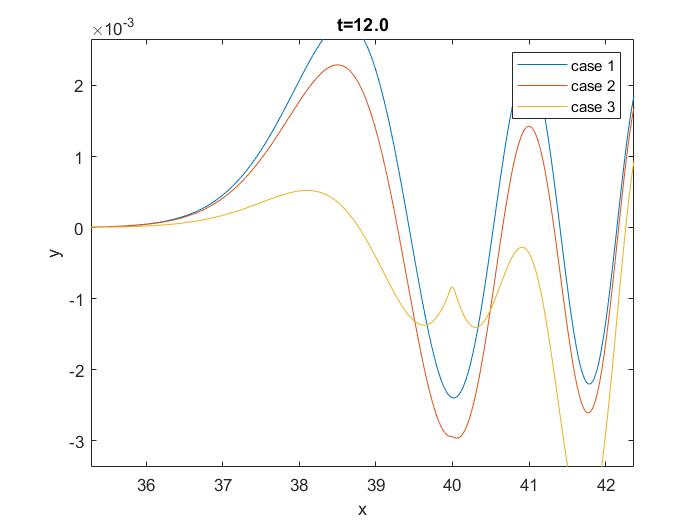}}
\end{minipage}
\caption{In these plots, we analyse behaviours of the solutions of the equation \eqref{RE-01} in the cases of different mass terms. In the upper-left plot, the graphic of the initial function $u_0$ is given. In the further plots, we compare the replacement function $u$ at $t=8.8, 10.2, 10.6, 11.0, 12.0$ for $\varepsilon=0.05$ in the three cases of the mass term, which are described below.} \label{fig1}
\end{figure}

For $m$ we consider the following cases, with $\delta$ denoting the standard Dirac's delta-distribution:
\begin{itemize}
  \item[Case 1:] $m(x)=0$ with $m_{\varepsilon}(x)=0$;
  \item[Case 2:] $m(x)=\delta(x-40)$ with $m_{\varepsilon}(x)=\varphi_\varepsilon(x-40)$;
  \item[Case 3:] $m(x)=\delta(x-40)\times\delta(x-40)$. Here, we understand $m_{\varepsilon}(x)$ as $m_{\varepsilon}(x)=\left(\varphi_\varepsilon(x-40)\right)^{2};$
\end{itemize}

In Figure \ref{fig1}, we analyse behaviours of the solutions to the equation \eqref{RE-01} with the initial function $u_0$ (given in the upper-left plot) in the cases of different mass terms. The further plots of Figure \ref{fig1} are comparing the replacement function $u$ at $t=8.8, 10.2, 10.6, 11.0, 12.0$ for $\varepsilon=0.05$ in the following three cases: Case 1 is corresponding to the mass term $m$ is equal to zero; Case 2 is corresponding to the case when the mass term $m$ is like a $\delta$-function; Case 3 is corresponding to the mass term $m$ is like a square of the $\delta$-function.

By analysing Figure \ref{fig1}, we see that a delta-function mass term affects less on the behaviour of the solution of \eqref{RE-01} compared to the square delta-function like mass term by reflecting some waves in the opposite direction. In the upper-right plot and in the lower plots of Figure \ref{fig1}, we observe that the replacement function $u$ is almost fully reflected in the square delta-function like mass term case. At $t=8.8$ we see that the yellow coloured wave is starting to settle and, from $t=10.2$ is moving in opposite direction. We call the last phenomena, a "wall effect".

All numerical computations are made in C++ by using the sweep method. In above numerical simulations, we use the Matlab R2018b. For all simulations we take $\Delta t=0.2$, $\Delta x=0.01.$

{\bf Conclusion.} The analysis conducted in this article shows that numerical methods work well in situations where a rigorous mathematical formulation of the problem is difficult in the framework of the classical theory of distributions. The concept of very weak solutions eliminates this difficulty in the case of the terms with multiplication of distributions. In contrast with the framework of the Colombeau algebras (see \cite{Obe92}) where the consistency with classical solutions maybe lost, the concept of very weak solutions which depends heavily on the equation under consideration is consistent with classical theory. In particular, in the Klein-Gordon equation case, we see that a delta-function mass term affects less on the behaviour of the waves compared to the square of the delta-function case, the latter causing a so-called "wall effect".

Numerical experiments have shown that the concept of very weak solutions is very suitable for numerical modelling. In addition, using the theory of very weak solutions, we can talk about the uniqueness of numerical solutions of differential equations with strongly singular coefficients in an appropriate sense.

Essentially, the present work can be considered as a generalization of the study of the Klein-Gordon equation by introducing the fractional Laplacian instead of the classical one and by considering a spatially dependent mass. Moreover, we are treating the case of singular masses which has been less investigated in the literature.solutions of differential equations with strongly singular coefficients in an appropriate sense.

\end{document}